\documentclass[12pt]{article}
\usepackage{amsmath,amssymb}

\newtheorem{theorem}{Theorem}[section]
\newtheorem{proposition}[theorem]{Proposition}
\newtheorem{corollary}[theorem]{Corollary}
\newtheorem{lemma}[theorem]{Lemma}

\newtheorem{remark}[theorem]{Remark}
\newtheorem{assumption}[theorem]{Assumption}

\newenvironment{proof}{\smallskip\par{\sc Proof.}\enspace}%
 {{\unskip\nobreak\hfil\penalty50\hskip2em
          \hbox{}\nobreak\hfil{\rule[-1pt]{5pt}{10pt}}
          \parfillskip=0pt\finalhyphendemerits=0
          \par\medskip}} 

\makeatletter
\def\section{\@startsection {section}{1}{\z@}{3.25ex plus 1ex minus
 .2ex}{1.5ex plus .2ex}{\large\bf}}
\def\subsection{\@startsection{subsection}{2}{\z@}{3.25ex plus 1ex minus
 .2ex}{1.5ex plus .2ex}{\normalsize\bf}}
\@addtoreset{equation}{section} 
\makeatother

\title{ Rate of convergence for Wong-Zakai-type approximations of It\^o stochastic
differential equations}

\author{Bilel Kacem Ben Ammou\footnote{Department of Mathematics, Universtiy of Tunis - El Manar, Street Mohamed Alaya Kacem,  Nabeul - Tunisia. E-mail: \emph{bilelbenammou@gmail.com}}\quad\quad Alberto Lanconelli\footnote{Dipartimento di Matematica, Universit\'a degli Studi di Bari Aldo Moro, Via E. Orabona 4, 70125 Bari - Italia. E-mail: \emph{alberto.lanconelli@uniba.it}}}

\date{\empty}

\begin{document}

\maketitle

\numberwithin{equation}{section}

\bigskip

\begin{abstract}
We consider a class of stochastic differential equations driven by a one dimensional Brownian motion and we investigate the rate of convergence for Wong-Zakai-type approximated solutions. We first consider the Stratonovich case, obtained through the point-wise multiplication between the diffusion coefficient and a smoothed version of the noise; then, we consider It\^o equations where the diffusion coefficient is Wick-multiplied by the regularized noise. We discover that in both cases the speed of convergence to the exact solution coincides with the speed of convergence of the smoothed noise towards the original Brownian motion. We also prove, in analogy with a well known property for exact solutions, that the solutions of approximated It\^o equations solve approximated Stratonovich equations with a certain correction term in the drift.
\end{abstract}



\allowdisplaybreaks

\section{Introduction and statement of the main results}

From a modeling point of view, the celebrated Wong-Zakai theorem \cite{WZ},\cite{WZ2} provides a crucial insight in the theory of stochastic differential equations. It asserts that the solution $\{X_t^{(n)}\}_{t\in [0,T]}$ of the random ordinary differential equation
\begin{eqnarray}\label{stra approx intro}
\frac{dX_t^{(n)}}{dt}=b(t,X_t^{(n)})+\sigma(t,X_t^{(n)})\cdot\frac{dB_t^{(n)}}{dt},
\end{eqnarray}
where $\{B_t^{(n)}\}_{t\in [0,T]}$ is a suitable smooth approximation of the Brownian motion $\{B_t\}_{t\in [0,T]}$, converges in the mean, as $n$ goes to infinity, to the solution of the Stratonovich stochastic differential equation (SDE, for short)
\begin{eqnarray}\label{stra intro}
dX_t=b(t,X_t)dt+\sigma(t,X_t)\circ dB_t.
\end{eqnarray}
At a first sight, it may look a bit surprising the fact that the sequence $\{X_t^{(n)}\}_{t\in [0,T]}$ does not converge to the It\^o's interpretation of the corresponding stochastic equation, i.e.
\begin{eqnarray}\label{ito intro}
dX_t=b(t,X_t)dt+\sigma(t,X_t)dB_t.
\end{eqnarray}
What makes the sequence $\{X_t^{(n)}\}_{t\in [0,T]}$ prefer to converge to the Stratonovich SDE (\ref{stra intro}) instead of the It\^o SDE (\ref{ito intro}) is the presence of the point-wise product $\cdot$ appearing in (\ref{stra approx intro}) between the diffusion coefficient $\sigma$ and the smoothed noise. In fact, Hu and {\O}ksendal \cite{HO} proved, when the diffusion coefficient is linear, that the solution of
\begin{eqnarray}\label{ito approx intro}
\frac{dY_t^{(n)}}{dt}=b(t,Y_t^{(n)})+\sigma(t)Y_t^{(n)}\diamond\frac{dB_t^{(n)}}{dt},
\end{eqnarray}
where $\diamond$ stands for the Wick product, converges as $n$ goes to infinity to the solution of the It\^o SDE
\begin{eqnarray}\label{ito oksendal intro}
dY_t=b(t,Y_t)dt+\sigma(t)Y_tdB_t.
\end{eqnarray}
Along this direction, Da Pelo et al. \cite{DLS 2013} introduced a family of products interpolating  between the point-wise and Wick products and proved convergence for Wong-Zakai-type approximations toward stochastic differential equations where the stochastic integrals are defined via suitable evaluation points in the Riemann sums.\\
Approximation procedures based on Wong-Zakai-type theorems have attracted the attention of several authors. First of all, Stroock and Varadhan \cite{Stroock Varadhan} proved the multidimensional version of the Wong-Zakai theorem. Then, generalizations to SDEs driven by different type of noises and to stochastic partial differential equations have been the most investigated directions. For instance, Konecny \cite{Konecny} proved a Wong-Zakai-type theorem for one-dimensional SDEs driven by a semimartingale, Gyo\"ngy and G. Michaletzky \cite{Gyongy Michaletzky} considered $\delta$-martingales while Naganuma \cite{Naganuma} examined the case of Gaussian rough paths. In the theory of stochastic partial differential equations, Hairer and Pardoux \cite{Hairer Pardoux} proved a version of the Wong-Zakai theorem for one-dimensional parabolic nonlinear stochastic PDEs driven by space-time white noise utilizing the recent theory of regularity structures; Brezniak and Flandoli \cite{Brezniak Flandoli} proved almost sure convergence to the solution to a Stratonovich stochastic partial differential equation; Tessitore and Zabczyk \cite{Tessitore Zabczyk} obtained results on the weak convergence of the laws of the Wong-Zakai approximations for stochastic evolution equations. We also mention that Londono and Villegas \cite{Londono Villegas} proposed to use a Wong-Zakai type approximation method for the numerical evaluation of the solutions of SDEs.\\
The aim of the present paper is to compare the rate of convergence for approximations of Stratonovich and It\^o quasi-linear SDEs and to investigate whether the connection between exact solutions of the two different interpretations can be restored for the corresponding approximating sequences (see the discussion after Corollary \ref{corollary} below).

We remark that the rate of convergence for Wong-Zakai approximations, in the Stratonovich case, has been already investigated by other authors. We recall Hu and Nualart \cite{HN} dealing with almost sure convergence in H\"older norms; Hu, Kallianpur and Xiong \cite{HKX} studying approximations for the Zakai equation and Gyongy and Shmatkov \cite{Gyongy Shmatkov} and Gyongy and Stinga \cite{Gyongy Stinga} treating general linear stochastic partial differential equations. We also refer the reader to the book by Hu \cite{H} where Wong-Zakai approximations are considered in the framework of Euler-Maruyama discretization schemes.We will discuss in Remark \ref{Hu} below the details of the comparison between our convergence rate for Stratonovich equations and the one in \cite{H}.

While Wong-Zakai-type theorems for Stratonovich SDEs have been largely investigated, approximations for It\^o SDEs are very rare in the literature. In fact, as the paper by Hu and {\O}ksendal shows, to recover the It\^o interpretation of the SDE one has to deal with the Wick product and in most cases this multiplication is not easy to handle. This is the reason why we focus on equations with linear diffusion coefficient (it is in fact not known whether the fully non linear version of (\ref{ito approx intro}) admits a solution \cite{HO}). Nevertheless, to find the speed of convergence of the approximation to the solution of the It\^o equation, we had to utilize some tools from the Malliavin calculus (see Lemma \ref{Malliavin} below). The present paper can be considered as a continuation of the work presented in Da Pelo et al. \cite{DLS 2013}, where the issue of the rate of convergence has not been studied.\\
To state our main results we briefly describe our framework. Let $(W,\mathcal{A},\mu)$ be the classical Wiener space over the time interval $[0,T]$, where $T$ is an arbitrary positive constant, and denote by $\{B_t\}_{t\in [0,T]}$ the coordinate process, i.e.
\begin{eqnarray*}
B_t:W&\to&\mathbb{R}\\
\omega&\mapsto& B_t(\omega)=\omega(t).
\end{eqnarray*}
By construction, the process $\{B_t\}_{t\in [0,T]}$ is, under the measure $\mu$, a one dimensional Brownian motion. We now introduce a smooth (continuously differentiable) approximation of $\{B_t\}_{t\in [0,T]}$ by means of a kernel satisfying certain technical assumptions. In the sequel, the symbol $|f|$ will denote the norm of $f\in L^2([0,T])$ while $\Vert X\Vert_p$ will denote the norm of $X\in\mathcal{L}^p(W,\mu)$ for any $p\geq 1$.
\begin{assumption}\label{assumption on K}
For any $\varepsilon>0$ let $K_{\varepsilon}:[0,T]^2\to\mathbb{R}$ be such that
\begin{itemize}
\item the function $t\mapsto K_{\varepsilon}(t,s)$ belongs to $C^1([0,T])$ for almost all $s\in [0,T]$;
\item the functions $s\mapsto K_{\varepsilon}(t,s)$ and $s\mapsto \partial_tK_{\varepsilon}(t,s)$ belong to $L^2([0,T])$ for all $t\in [0,T]$.
\end{itemize}
Moreover, we assume that
\begin{eqnarray}\label{kernel}
\lim_{\varepsilon\to 0^+}\sup_{t\in[0,T]}|K_{\varepsilon}(t,\cdot)-1_{[0,t]}(\cdot)|=0
\end{eqnarray}
and
\begin{eqnarray*}
M:=\sup_{\varepsilon >0}\sup_{t\in [0,T]}|K_{\varepsilon}(t,\cdot)|<+\infty.
\end{eqnarray*}
\end{assumption}
Now, if we let
\begin{eqnarray*}
B^{\varepsilon}_t:=\int_0^TK_{\varepsilon}(t,s)dB_s,\quad t\in [0,T],
\end{eqnarray*}
and recall that $B_t=\int_0^T1_{[0,t]}(s)dB_s$, then Assumption \ref{assumption on K} implies that $\{B^{\varepsilon}_t\}_{t\in [0,T]}$ is a continuosly differentiable Gaussian process and that
$B_t^{\varepsilon}$ converges to $B_t$ in $\mathcal{L}^2(W,\mu)$ uniformly with respect to $t\in [0,T]$. In fact, condition (\ref{kernel}) is equivalent to
\begin{eqnarray*}
\lim_{\varepsilon\to 0^+}\sup_{t\in [0,T]}\Vert B_t^{\varepsilon}-B_t\Vert_2=0.
\end{eqnarray*}
Therefore, we deal with a quite general class of smooth approximations of the Brownian motion $\{B_t\}_{t\geq 0}$. In the sequel we will be studying SDEs of the type (\ref{ito oksendal intro}) both in the Stratonovich and It\^o senses. We now state the assumptions on the coefficients $b$ and $\sigma$ which are supposed to be valid for the rest of the present paper.
\begin{assumption}\label{assumption on b and sigma}
There exist two positive constants $C_1$ and $C_2$ such that for all $t\in [0,T]$ and $x,y\in\mathbb{R}$ one has
\begin{eqnarray}\label{lipschitz}
|b(t,x)-b(t,y)|\leq C_1|x-y|\quad\mbox{ and }\quad |b(t,x)|\leq C_2(1+|x|).
\end{eqnarray}
Moreover, the function $\sigma$ belongs to $\mathcal{L}^{\infty}([0,T])$.
\end{assumption}
For $f\in L^2([0,T])$ we denote
\begin{eqnarray*}
\mathcal{E}(f):=\exp\left\{\int_0^Tf(s)dB_s-\frac{1}{2}\int_0^Tf^2(s)ds\right\}
\end{eqnarray*}
and we call it \emph{stochastic exponential}. The set $\{\mathcal{E}(f),f\in L^2([0,T])\}$ turns out to be total in $\mathcal{L}^p(W,\mu)$ for any $p\geq 1$. Given $f,g\in L^2([0,T])$, the \emph{Wick product} of $\mathcal{E}(f)$ and $\mathcal{E}(g)$ is defined to be
\begin{eqnarray*}
\mathcal{E}(f)\diamond\mathcal{E}(g):=\mathcal{E}(f+g).
\end{eqnarray*}
This multiplication can be extended by linearity and density to an unbounded bilinear form on a proper subset of $\mathcal{L}^p(W,\mu)\times\mathcal{L}^p(W,\mu)$ (see Holden et al. \cite{HOUZ} and Janson \cite{J} for its connection to It\^o-Skorohod integration theory). For $g\in L^2([0,T])$ we also define the \emph{translation operator} $T_g$ as the operator that shifts the Brownian path by the function $\int_0^{\cdot}g(s)ds$; more precisely, the action of $T_g$ on stochastic exponentials is given by
\begin{eqnarray*}
T_g\mathcal{E}(f):=\mathcal{E}(f)\cdot\exp\{\langle f,g\rangle\}.
\end{eqnarray*}
where $\langle\cdot,\cdot\rangle$ denotes the inner product in $L^2([0,T])$ (see Holden et al. \cite{HOUZ} and Janson \cite{J} for details).\\
\noindent We are now ready to state the first two main theorems of the present paper. The proofs are postponed to Section 2 and Section 3, respectively.

\begin{theorem}\label{main theorem 1}
Let $\{X_t\}_{t\in [0,T]}$ be the unique solution of the Stratonovich SDE
\begin{eqnarray}\label{stra SDE}
dX_t=b(t,X_t)dt+\sigma(t)X_t\circ dB_t,\quad t\in ]0,T]\quad\quad X_0=x
\end{eqnarray}
and for any $\varepsilon>0$ let $\{X_t^{\varepsilon}\}_{t\in[0,T]}$ be the unique solution of
\begin{eqnarray}\label{approx stra}
\frac{dX_t^{\varepsilon}}{dt}=b(t,X_t^{\varepsilon})+\sigma(t)X_t^{\varepsilon}\cdot\frac{dB_t^{\varepsilon}}{dt},\quad
X_0^{\varepsilon}=x.
\end{eqnarray}
Then, for any $p\geq 1$ there exists a positive constant $C$ (depending on $p$, $|x|$, $T$, $C_1$, $C_2$ and $M$) such that for any $q$ greater than $p$
\begin{eqnarray}\label{SDE1}
\sup_{t\in [0,T]}\Vert X_t^{\varepsilon}-X_t\Vert_{p}\leq C\cdot\mathcal{S}_q\left(\sup_{t\in [0,T]}|K_{\varepsilon}(t,\cdot)-1_{[0,t]}(\cdot)|\right)
\end{eqnarray}
where
\begin{eqnarray}\label{def S}
\mathcal{S}_q(\lambda):=\lambda\exp\left\{q\lambda^2\right\}+\exp\{\lambda^2/2\}-1,\quad \lambda\in\mathbb{R}
\end{eqnarray}
\end{theorem}

\begin{remark}\label{Hu}
In Theorem 11.6 of \cite{H} it is proved that
\begin{eqnarray}\label{Hu estimate}
\Big\Vert  \sup_{t\in [0,T]}|X_t^{\pi}-X_t|\Big\Vert_p\leq C_{p,T}(\log|\pi|)^2|\pi|^{\frac{1}{2}+\frac{1}{\log|\pi|}}
\end{eqnarray}
where $\pi$ is a partition of the interval $[0,T]$, $|\pi|$ denotes the mesh of the partition $\pi$ and $\{X_t^{\pi}\}_{t\in [0,T]}$ stands for the solution of
\begin{eqnarray*}\label{Hu SDE}
\frac{dX_t^{\pi}}{dt}=b(t,X_t^{\pi})+\sigma(t,X_t^{\pi})\cdot\frac{dB_t^{\pi}}{dt},
\end{eqnarray*}
with $\{B_t^{\pi}\}_{t\in [0,T]}$ being the polygonal approximation of $\{B_t\}_{t\in [0,T]}$ associated to the partition $\pi$. The above result is stated and proved for general nonlinear systems of SDEs driven by a multidimensional Brownian motion. Moreover, the topology utilized in (\ref{Hu estimate}) is stronger than the one in (\ref{SDE1}) (where the supremum is outside of the $\mathcal{L}^p(W,\mu)$-norm).\\
It is not difficult to see that the polygonal approximation $\{B^{\pi}_t\}_{t\in [0,T]}$ is included in the family of approximations $\{B^{\varepsilon}_t\}_{t\in [0,T]}$ considered in the present paper (the parameter $\varepsilon$ reduces to the mesh of the partition $|\pi|$) and in that case we get
\begin{eqnarray*}
\sup_{t\in [0,T]}\Vert B_t^{\pi}-B_t\Vert_2=\sup_{t\in [0,T]}|K_{|\pi|}(t,\cdot)-1_{[0,t]}(\cdot)|=C\sqrt{|\pi|}.
\end{eqnarray*}
Substituting in (\ref{SDE1}) we obtain
\begin{eqnarray*}
\sup_{t\in [0,T]}\Vert X_t^{\pi}-X_t\Vert_{p}\leq C\cdot\mathcal{S}_q(C\sqrt{|\pi|})
\end{eqnarray*}
which behaves like $\sqrt{|\pi|}$  for $|\pi|$ going to zero. A comparison with (\ref{Hu estimate}) shows that Theorem \ref{main theorem 1} provides a highest rate of convergence, at the price of a weaker topology and more restrictive conditions on the class of the SDEs considered.\\
The result and proof of Theorem \ref{main theorem 1} are however necessary for the comparison proposed in the present paper.
\end{remark}

\begin{theorem}\label{main theorem 2}
Let $\{Y_t\}_{t\in [0,T]}$ be the unique solution of the It\^o SDE
\begin{eqnarray}\label{ito SDE}
dY_t=b(t,Y_t)dt+\sigma(t)Y_tdB_t,\quad t\in ]0,T]\quad\quad Y_0=x
\end{eqnarray}
and for any $\varepsilon>0$ let $\{Y_t^{\varepsilon}\}_{t\in[0,T]}$ be the unique solution of
\begin{eqnarray}\label{approx ito}
\frac{dY_t^{\varepsilon}}{dt}=b(t,Y_t^{\varepsilon})+\sigma(t)Y_t^{\varepsilon}\diamond\frac{dB_t^{\varepsilon}}{dt},\quad
Y_0^{\varepsilon}=x.
\end{eqnarray}
Then, for any $p\geq 1$ there exists a positive constant $C$ (depending on $p$, $|x|$, $T$, $C_1$, $C_2$ and $M$) such that for any $q$ greater than $p$
\begin{eqnarray*}
\sup_{t\in [0,T]}\Vert Y_t^{\varepsilon}-Y_t\Vert_{p}\leq C\cdot\mathcal{S}_q\left(\sqrt{2}\sup_{t\in [0,T]}|K_{\varepsilon}(t,\cdot)-1_{[0,t]}(\cdot)|\right)
\end{eqnarray*}
where $\mathcal{S}$ is the function defined in (\ref{def S}).
\end{theorem}

\begin{corollary}\label{corollary}
In the notation of Theorem \ref{main theorem 1} and Theorem \ref{main theorem 2}, we have for any $p\geq 1$ that
\begin{eqnarray*}
\lim_{\varepsilon\to 0^+}\sup_{t\in [0,T]}\Vert X_t^{\varepsilon}-X_t\Vert_{p}=\lim_{\varepsilon\to 0^+}\sup_{t\in [0,T]}\Vert Y_t^{\varepsilon}-Y_t\Vert_{p}=0
\end{eqnarray*}
where both limits have rate of convergence of order
\begin{eqnarray*}
\sup_{t\in [0,T]}|K_{\varepsilon}(t,\cdot)-1_{[0,t]}(\cdot)|\quad\mbox{ as $\varepsilon$ tends to zero}.
\end{eqnarray*}
\end{corollary}

\begin{proof}
It follows from
\begin{eqnarray}\label{limit of S}
\lim_{\lambda\to 0^+}\frac{\mathcal{S}_q(\lambda)}{\lambda}=1.
\end{eqnarray}
for all $q\geq 1$.
\end{proof}

\noindent It is well known (see for instance Karatzas and Shreve \cite{KS}) that the It\^o SDE (\ref{ito SDE}) can be reformulated as the Stratonovich SDE
\begin{eqnarray}\label{ito-stra SDE}
dY_t=\left(b(t,Y_t)-\frac{1}{2}\sigma(t)Y_t\right)dt+\sigma(t)Y_t\circ dB_t,\quad t\in ]0,T]\quad\quad Y_0=x.
\end{eqnarray}
The next theorem provides a similar representation for the approximated It\^o equation (\ref{approx ito}) in terms of a suitable approximated Stratonovich equation. The proof can be found in Section 4.

\begin{theorem}\label{main theorem 3}
For any $\varepsilon>0$ let $\{Y_t^{\varepsilon}\}_{t\in[0,T]}$ be the unique solution of
\begin{eqnarray*}
\frac{dY_t^{\varepsilon}}{dt}=b(t,Y_t^{\varepsilon})+\sigma(t)Y_t^{\varepsilon}\diamond\frac{dB_t^{\varepsilon}}{dt},\quad
X_0^{\varepsilon}=x.
\end{eqnarray*}
Then, for any $t\in [0,T]$ we have
\begin{eqnarray}\label{ito}
Y_{t}^{\varepsilon}=T_{-K_{\varepsilon}(t,\cdot)}S_{t}^{\varepsilon},
\end{eqnarray}
where $\{S_{t}^{\varepsilon}\}_{t\in[0,T]}$ is the unique solution of
\begin{eqnarray*}
\frac{dS_{t}^{\varepsilon}}{dt}=b(t,S_{t}^{\varepsilon})+\frac{1}{2}\frac{d|K_{\varepsilon}(t,\cdot)|^{2}}{dt}\cdot S_{t}^{\varepsilon}+ \sigma(t)S_{t}^{\varepsilon}\cdot\frac{dB_{t}^{\varepsilon}}{dt},\quad S_{0}^{\varepsilon}=x.
\end{eqnarray*}
\end{theorem}

\noindent The paper is organized as follows: Section 2 and Section 3 are devoted to the proofs of Theorem \ref{main theorem 1} and Theorem \ref{main theorem 2}, respectively. Both sections also contain some preliminary results and estimates utilized in the proofs of the main results, which are divided in two major steps. Section 4 contains two different proofs of Theorem \ref{main theorem 3}, the second one being a direct verification of the identity (\ref{ito}).

\section{Proof of Theorem \ref{main theorem 1}}

\subsection{Auxiliary results and remarks: Stratonovich case}
The proof of Theorem \ref{main theorem 1} will be carried for the simplified equation where $b$ does not depend on $t$ and $\sigma$ is identically equal to one. Straightforward modifications will lead to the general case.\\ \noindent The existence and uniqueness for the solutions of (\ref{approx stra}) and (\ref{stra SDE}) follow, in view of Assumption \ref{assumption on K} and Assumption \ref{assumption on b and sigma}, by standard results in the theory of stochastic and ordinary differential equations. We also refer the reader to Theorem 5.5 in \cite{DLS 2013} for a proof using the techniques adopted in this paper. To ease the notation we define
\begin{eqnarray*}
E_{\varepsilon}(t):=\exp\{\delta(-K_{\varepsilon}(t,\cdot))\}\quad\mbox{ and }\quad E_{0}(t):=\exp\{\delta(-1_{[0,t]}(\cdot))\}.
\end{eqnarray*}
Here and in the sequel the symbol $\delta(f)$ stands for $\int_0^Tf(s)dB_s$. We begin by observing that (see the proof of Theorem 5.5 in \cite{DLS 2013}) the solution $\{X_t^{\varepsilon}\}_{t\in[0,T]}$ from Theorem \ref{main theorem 1} can be represented as
\begin{eqnarray*}
X_t^{\varepsilon}=Z_t^{\varepsilon}\cdot E_{\varepsilon}(t)^{-1}
\end{eqnarray*}
where
\begin{eqnarray*}
\frac{dZ_t^{\varepsilon}}{dt}=b(Z_t^{\varepsilon}\cdot E^{-1}_{\varepsilon}(t))\cdot E_{\varepsilon}(t),\quad Z_0^{\varepsilon}=x.
\end{eqnarray*}
The same holds true for $\{X_t\}_{t\in[0,T]}$; more precisely,
\begin{eqnarray*}
X_t=Z_t\cdot E_{0}(t)^{-1}
\end{eqnarray*}
where
\begin{eqnarray*}
\frac{dZ_t}{dt}=b(Z_t\cdot E^{-1}_{0}(t))\cdot E_{0}(t),\quad Z_0=x.
\end{eqnarray*}
Moreover, we have the estimates
\begin{eqnarray*}
|Z_t^{\varepsilon}|&\leq&|x|+\int_0^t|b(Z_s^{\varepsilon}\cdot E_{\varepsilon}(s)^{-1})\cdot E_{\varepsilon}(s)|ds\\
&\leq&|x|+\int_0^tC_2\left(1+|Z_s^{\varepsilon}\cdot E_{\varepsilon}(s)^{-1}|\right)\cdot E_{\varepsilon}(s)ds\\
&=&|x|+\int_0^tC_2E_{\varepsilon}(s)ds+\int_0^tC_2|Z_s^{\varepsilon}|ds\\
&\leq& |x|+\int_0^TC_2E_{\varepsilon}(s)ds+\int_0^tC_2|Z_s^{\varepsilon}|ds.
\end{eqnarray*}
By the Gronwall inequality,
\begin{eqnarray*}
|Z_t^{\varepsilon}|\leq\left(|x|+\int_0^T C_2E_{\varepsilon}(s)ds \right)e^{C_2t}.
\end{eqnarray*}
This shows that for any $q\geq 1$ we have the bound
\begin{eqnarray}\label{estimate norm of Z}
\left\|\sup_{t\in [0,T]}|Z_t^{\varepsilon}|\right\|_q&\leq&\left(|x|+\int_0^T C_2\Vert E_{\varepsilon}(s)\Vert_qds \right)e^{C_2T}\nonumber\\
&=&\left(|x|+\int_0^T C_2\exp\left\{\frac{q}{2}|K_{\varepsilon}(s,\cdot)|^2\right\}ds \right)e^{C_2T}\nonumber\\
&\leq&\left(|x|+C_2 T\exp\left\{\frac{q}{2}\sup_{s\in[0,T]}|K_{\varepsilon}(s,\cdot)|^2\right\} \right)e^{C_2T}.
\end{eqnarray}

\noindent To prove Theorem \ref{main theorem 1} we need the following estimate which is of independent interest.

\begin{proposition}\label{distance exponentials}
Let $f,g\in L^2([0,T])$. Then, for any $p\geq 1$ we have
\begin{eqnarray*}
\Vert\exp\{\delta(f)\}-\exp\{\delta(g)\}\Vert_p\leq C\mathcal{S}_p(|f-g|)
\end{eqnarray*}
where
\begin{eqnarray*}
\mathcal{S}_p(\lambda):=\lambda\exp\left\{p\lambda^2\right\}+\exp\{\lambda^2/2\}-1,\quad \lambda\in\mathbb{R}
\end{eqnarray*}
and $C$ is a constant depending on $p$ and $|g|$.
\end{proposition}

\begin{proof}
The proof involves few notions of Malliavin calculus. We refer the reader to the books of Nualart \cite{Nualart} and Bogachev \cite{Bogachev}. Let $f\in L^2([0,T])$ and $p\geq 1$; then, according to the Poincar\'e inequality (see Theorem 5.5.11 in Bogachev \cite{Bogachev}), we can write
\begin{eqnarray*}
\Vert\exp\{\delta(f)\}-1\Vert_p&\leq&\Vert\exp\{\delta(f)\}-E[\exp\{\delta(f)\}]\Vert_p+|E[\exp\{\delta(f)\}]-1|\\
&=&\Vert\exp\{\delta(f)\}-\exp\{|f|^2/2\}]\Vert_p+\exp\{|f|^2/2\}-1\\
&\leq&\mathcal{C}(p)\left\| |D\exp\{\delta(f)\}|_{L^2([0,T])}\right\|_p+\exp\{|f|^2/2\}-1\\
&=&\mathcal{C}(p)\left\| |\exp\{\delta(f)\}f|_{L^2([0,T])}\right\|_p+\exp\{|f|^2/2\}-1\\
&=&\mathcal{C}(p)|f|\Vert\exp\{\delta(f)\}\Vert_p+\exp\{|f|^2/2\}-1\\
&=&\mathcal{C}(p)|f|\exp\left\{\frac{p}{2}|f|^2\right\}+\exp\{|f|^2/2\}-1
\end{eqnarray*}
where $D$ denotes the Malliavin derivative and $\mathcal{C}(p)$ is a positive constant depending only on $p$. Therefore, for any $f,g\in L^2([0,T])$ and $p\geq 1$ we have
\begin{eqnarray*}
\Vert\exp\{\delta(f)\}-\exp\{\delta(g)\}\Vert_p&=&\Vert\exp\{\delta(g)\}\left(\exp\{\delta(f-g)\}-1\right)\Vert_p\\
&\leq&\Vert\exp\{\delta(g)\}\Vert_{2p}\Vert\exp\{\delta(f-g)\}-1\Vert_{2p}\\
&\leq&e^{p|g|^2}\left(\mathcal{C}(2p)|f-g|\exp\left\{p|f-g|^2\right\}+\exp\{|f-g|^2/2\}-1\right)\\
&\leq&C\mathcal{S}_p(|f-g|)
\end{eqnarray*}
where we utilized the H\"older inequality.
\end{proof}

\subsection{Proof of Theorem \ref{main theorem 1}}

The proof is divided in two steps.\\

\noindent \textbf{Step one}: We prove that for any $p\geq 1$ there exists a positive constant $C$ (depending on $p$, $|x|$, $T$, $C_1$, $C_2$ and $M$) such that for any $q$ greater than $p$
\begin{eqnarray}\label{inequality step 1}
\left\|\sup_{t\in [0,T]} |Z_t^{\varepsilon}-Z_t|\right\|_p\leq C\cdot\mathcal{S}_{q}\left(\sup_{s\in [0,T]}|K_{\varepsilon}(s,\cdot)-1_{[0,s]}(\cdot)|\right)
\end{eqnarray}
We begin by using the equations solved by $Z_t^{\varepsilon}$ and $Z_t$ and the assumptions on $b$ to get
\begin{eqnarray*}
|Z_t^{\varepsilon}-Z_t|&=&\Big|\int_0^tb(Z_s^{\varepsilon}E_{\varepsilon}(s)^{-1})E_{\varepsilon}(s)ds-\int_0^tb(Z_sE_0(s)^{-1})E_0(s)ds\Big|\\
&\leq&\Big|\int_0^tb(Z_s^{\varepsilon}E_{\varepsilon}(s)^{-1})E_{\varepsilon}(s)-b(Z_sE_0(s)^{-1})E_{\varepsilon}(s)ds\Big|\\
&&+\Big|\int_0^tb(Z_sE_0(s)^{-1})E_{\varepsilon}(s)-b(Z_sE_0(s)^{-1})E_0(s)ds\Big|\\
&\leq&\int_0^t|b(Z_s^{\varepsilon}E_{\varepsilon}(s)^{-1})-b(Z_sE_0(s)^{-1})|E_{\varepsilon}(s)ds\\
&&+\int_0^t|b(Z_sE_0(s)^{-1})||E_{\varepsilon}(s)-E_0(s)|ds\\
&\leq&\int_0^tC_1|Z_s^{\varepsilon}E_{\varepsilon}(s)^{-1}-Z_sE_0(s)^{-1}|E_{\varepsilon}(s)ds\\
&&+\int_0^tC_2(1+|Z_sE_0(s)^{-1}|)|E_{\varepsilon}(s)-E_0(s)|ds\\
&\leq&C_1\int_0^t|Z_s^{\varepsilon}E_{\varepsilon}(s)^{-1}-Z_sE_{\varepsilon}(s)^{-1}|E_{\varepsilon}(s)+|Z_sE_{\varepsilon}(s)^{-1}-Z_sE_0(s)^{-1}|E_{\varepsilon}(s)ds\\
&&+C_2\int_0^t(1+|Z_s|E_0(s)^{-1})|E_{\varepsilon}(s)-E_0(s)|ds\\
&=&C_1\int_0^t|Z_s^{\varepsilon}-Z_s|ds+C_1\int_0^t|Z_s||E_{\varepsilon}(s)^{-1}-E_0(s)^{-1}|E_{\varepsilon}(s)ds\\
&&+C_2\int_0^t(1+|Z_s|E_0(s)^{-1})|E_{\varepsilon}(s)-E_0(s)|ds\\
&\leq&C_1\int_0^t|Z_s^{\varepsilon}-Z_s|ds+C_1\int_0^T|Z_s||E_{\varepsilon}(s)^{-1}-E_0(s)^{-1}|E_{\varepsilon}(s)ds\\
&&+C_2\int_0^T(1+|Z_s|E_0(s)^{-1})|E_{\varepsilon}(s)-E_0(s)|ds\\
&=&\Lambda_{\varepsilon}+C_1\int_0^t|Z_s^{\varepsilon}-Z_s|ds
\end{eqnarray*}
where
\begin{eqnarray*}
\Lambda_{\varepsilon}&:=&C_1\int_0^T|Z_s||E_{\varepsilon}(s)^{-1}-E_0(s)^{-1}|E_{\varepsilon}(s)ds\\
&&+C_2\int_0^T(1+|Z_s|E_0(s)^{-1})|E_{\varepsilon}(s)-E_0(s)|ds.
\end{eqnarray*}
By the Gronwall inequality we deduce that
\begin{eqnarray*}
|Z_t^{\varepsilon}-Z_t|\leq\Lambda_{\varepsilon}e^{C_1t},\quad t\in[0,T]
\end{eqnarray*}
and hence for $p\geq 1$ the inequality
\begin{eqnarray*}
\left\|\sup_{t\in [0,T]} |Z_t^{\varepsilon}-Z_t|\right\|_p\leq e^{C_1T}\Vert\Lambda_{\varepsilon}\Vert_p.
\end{eqnarray*}
We now estimate $\Vert\Lambda_{\varepsilon}\Vert_p$ by writing $\Lambda_{\varepsilon}=\Lambda_{\varepsilon}^1+\Lambda_{\varepsilon}^2$ where
\begin{eqnarray*}
\Lambda_{\varepsilon}^1:=C_1\int_0^T|Z_s||E_{\varepsilon}(s)^{-1}-E_0(s)^{-1}|E_{\varepsilon}(s)ds
\end{eqnarray*}
and
\begin{eqnarray*}
\Lambda_{\varepsilon}^2:=C_2\int_0^T(1+|Z_s|E_0(s)^{-1})|E_{\varepsilon}(s)-E_0(s)|ds.
\end{eqnarray*}
Applying the triangle and H\"{o}lder inequalities we get
\begin{eqnarray*}
\Vert\Lambda_{\varepsilon}^1\Vert_p&\leq&C_1\int_0^T\Vert|Z_s||E_{\varepsilon}(s)^{-1}-E_0(s)^{-1}|E_{\varepsilon}(s)\Vert_pds\\
&\leq&C_1\int_0^T\Vert
Z_s\Vert_{p_1}\Vert E_{\varepsilon}(s)^{-1}-E_0(s)^{-1}\Vert_{p_2}\Vert
E_{\varepsilon}(s)\Vert_{p_3}ds
\end{eqnarray*}
where $p_1, p_2, p_3\in [1,+\infty[$ satisfy $\frac{1}{p_1}+\frac{1}{p_2}+\frac{1}{p_3}=\frac{1}{p}$. From the estimate (\ref{estimate norm of Z}) and the identity
\begin{eqnarray*}
\Vert E_{\varepsilon}(s)\Vert_{p_3}=\exp\left\{\frac{p_3}{2}|K_{\varepsilon}(s,\cdot)|^2\right\}
\end{eqnarray*}
we can write
\begin{eqnarray*}
\Vert\Lambda_{\varepsilon}^1\Vert_p\leq C \int_0^T\Vert E_{\varepsilon}(s)^{-1}-E_0(s)^{-1}\Vert_{p_2}ds
\end{eqnarray*}
where $C$ denotes a positive constant depending on $C_1$, $C_2$, $|x|$, $T$, $p$ and $M$ (in the sequel $C$ will denote a generic constant, depending on the previously specified parameters, which may vary from one line to another). Moreover, employing Proposition \ref{distance exponentials} with $f(\cdot)=K_{\varepsilon}(s,\cdot)$ and $g(\cdot)=1_{[0,s]}(\cdot)$ we conclude that
\begin{eqnarray}\label{estimate for lambda_1}
\Vert\Lambda_{\varepsilon}^1\Vert_p&\leq& C \int_0^T\mathcal{S}_{p_2}(|K_{\varepsilon}(s,\cdot)-1_{[0,s]}(\cdot)|)ds\nonumber\\
&\leq&C\cdot\mathcal{S}_{p_2}\left(\sup_{s\in [0,T]}|K_{\varepsilon}(s,\cdot)-1_{[0,s]}(\cdot)|\right).
\end{eqnarray}
Note that for any $p\geq 1$ the function $\lambda\mapsto\mathcal{S}(\lambda)$ is increasing on $[0,+\infty]$.
Let us now consider $\Lambda_{\varepsilon}^2$; if we apply one more time the triangle and H\"{o}lder inequalities, then we get
\begin{eqnarray*}
\Vert\Lambda_{\varepsilon}^2\Vert_p&\leq&C_2\int_0^T\Vert(1+|Z_s|E_0(s)^{-1})|E_{\varepsilon}(s)-E_0(s)|\Vert_pds\\
&\leq&C_2\int_0^T\Vert1+|Z_s|E_0(s)^{-1}\Vert_{q_1}\Vert E_{\varepsilon}(s)-E_0(s)\Vert_{q_2}ds\\
&\leq&C_2\int_0^T(1+\Vert|Z_s|E_0(s)^{-1}\Vert_{q_1})\Vert E_{\varepsilon}(s)-E_0(s)\Vert_{q_2}ds
\end{eqnarray*}
where $q_1,q_2\in [1,+\infty[$ satisfy $\frac{1}{q_1}+\frac{1}{q_2}=\frac{1}{p}$. We observe that
\begin{eqnarray}\label{rhs}
1+\Vert|Z_s|E_0(s)^{-1}\Vert_{q_1}&\leq&1+\Vert Z_s\Vert_{r_1}\cdot\Vert E_0(s)^{-1}\Vert_{r_2}
\end{eqnarray}
where $\frac{1}{q_1}=\frac{1}{r_1}+\frac{1}{r_2}$ and that, according to estimate (\ref{estimate norm of Z}), the right hand side of (\ref{rhs}) is bounded uniformly in $s\in [0,T]$ by a constant $C$ depending on $C_1$, $C_2$, $|x|$, $T$, $p$ and $M$. Therefore,
\begin{eqnarray}\label{estimate for lambda_2}
\Vert\Lambda_{\varepsilon}^2\Vert_p&\leq& C\cdot\mathcal{S}_{q_2}\left(\sup_{s\in [0,T]}|K_{\varepsilon}(s,\cdot)-1_{[0,s]}(\cdot)|\right).
\end{eqnarray}
Here we utilized Proposition \ref{distance exponentials} with $f(\cdot)=-K_{\varepsilon}(s,\cdot)$ and $g(\cdot)=-1_{[0,s]}(\cdot)$. Finally, combining (\ref{estimate for lambda_1}) with (\ref{estimate for lambda_2}) we obtain
\begin{eqnarray*}
\left\|\sup_{t\in [0,T]} |Z_t^{\varepsilon}-Z_t|\right\|_p\leq C\cdot\mathcal{S}_p\left(\sup_{s\in [0,T]}|K_{\varepsilon}(s,\cdot)-1_{[0,s]}(\cdot)|^2\right).
\end{eqnarray*}

\noindent\textbf{Step two}: We prove that for any $p\geq 1$ there exists a positive constant $C$ (depending on $p$, $|x|$, $T$, $C_1$, $C_2$ and $M$) such that for any $q$ greater than $p$
\begin{eqnarray*}
\sup_{t\in [0,T]}\left\|X_t^{\varepsilon}-X_t\right\|_p\leq C\cdot\mathcal{S}_q\left(\sup_{s\in [0,T]}|K_{\varepsilon}(s,\cdot)-1_{[0,s]}(\cdot)|\right).
\end{eqnarray*}
We first note that
\begin{eqnarray*}
X_t^{\varepsilon}-X_t&=&Z_t^{\varepsilon}\cdot E_{\varepsilon}(t)^{-1}-Z_t\cdot E_{0}(t)^{-1}\\
&=&Z_t^{\varepsilon}\cdot E_{\varepsilon}(t)^{-1}-Z^{\varepsilon}_t\cdot E_{0}(t)^{-1}+Z_t^{\varepsilon}\cdot E_{0}(t)^{-1}-Z_t\cdot E_{0}(t)^{-1}\\
&=&Z_t^{\varepsilon}\cdot (E_{\varepsilon}(t)^{-1}-E_{0}(t)^{-1})+(Z_t^{\varepsilon}-Z_t)\cdot E_{0}(t)^{-1}.
\end{eqnarray*}
Now we take $p\geq 1$ and apply the triangle and H\"older inequalities to get
\begin{eqnarray*}
\Vert X_t^{\varepsilon}-X_t\Vert_{p}&\leq&\Vert Z_t^{\varepsilon}\Vert_{p_1}\cdot \Vert E_{\varepsilon}(t)^{-1}-E_{0}(t)^{-1}\Vert_{p_2}+\Vert Z_t^{\varepsilon}-Z_t\Vert_{q_1}\cdot \Vert E_{0}(t)^{-1}\Vert_{q_2}
\end{eqnarray*}
where $\frac{1}{p}=\frac{1}{p_1}+\frac{1}{p_2}=\frac{1}{q_1}+\frac{1}{q_2}$. From estimate (\ref{estimate norm of Z}) we know that $\Vert Z_t^{\varepsilon}\Vert_{p_1}$ is bounded uniformly in $t\in [0,T]$ for any $p_1\geq 1$ while Proposition \ref{distance exponentials} ensures that
\begin{eqnarray*}
\Vert E_{\varepsilon}(t)^{-1}-E_{0}(t)^{-1}\Vert_{p_2}\leq C\mathcal{S}_{p_2}\left(|K_{\varepsilon}(t,\cdot)-1_{[0,t]}(\cdot)|\right)
\end{eqnarray*}
with a constant independent of $t\in [0,T]$. Moreover, inequality (\ref{inequality step 1}) from \emph{Step one} gives for $r>q_1$
\begin{eqnarray*}
\Vert Z_t^{\varepsilon}-Z_t\Vert_{q_1}\leq C\cdot\mathcal{S}_r\left(\sup_{s\in [0,T]}|K_{\varepsilon}(s,\cdot)-1_{[0,s]}(\cdot)|\right).
\end{eqnarray*}
These last assertions imply
\begin{eqnarray*}
\sup_{t\in [0,T]}\Vert X_t^{\varepsilon}-X_t\Vert_{p}\leq C\cdot\mathcal{S}_q\left(\sup_{s\in [0,T]}|K_{\varepsilon}(s,\cdot)-1_{[0,s]}(\cdot)|\right).
\end{eqnarray*}
The proof is complete.

\section{Proof of Theorem \ref{main theorem 2}}

\subsection{Auxiliary results and remarks: It\^o case}

The proof of Theorem \ref{main theorem 2} will be carried for the simplified equation where $b$ does not depend on $t$ and $\sigma$ is identically equal to one. Straightforward modifications will lead to the general case. To ease the notation, we denote for $t\in [0,T]$
\begin{eqnarray*}
\mathcal{E}_{\varepsilon}(t):=\mathcal{E}(-K_{\varepsilon}(t,\cdot))=\exp\left\{\delta(-K_{\varepsilon}(t,\cdot))-\frac{1}{2}|K_{\varepsilon}(t,\cdot)|^2\right\}
\end{eqnarray*}
and
\begin{eqnarray*}
\mathcal{E}_{0}(t):=\mathcal{E}(-1_{[0,t]})=\exp\left\{\delta(-1_{[0,t]}(\cdot))-\frac{t}{2}\right\}.
\end{eqnarray*}
The existence and uniqueness for the solutions of (\ref{approx ito}) and (\ref{ito SDE}) can be found in Theorem 5.5 from \cite{DLS 2013}. There it was observed that the solution $\{Y_t^{\varepsilon}\}_{t\in[0,T]}$ of (\ref{approx ito}) can be represented as
\begin{eqnarray*}
Y_t^{\varepsilon}=V_t^{\varepsilon}\diamond \mathcal{E}_{\varepsilon}(t)^{\diamond -1}
\end{eqnarray*}
where
\begin{eqnarray}\label{equation for V}
\frac{dV_t^{\varepsilon}}{dt}=b\left(V_t^{\varepsilon}\diamond (\mathcal{E}_{\varepsilon}(t))^{\diamond -1}\right)\diamond \mathcal{E}_{\varepsilon}(t),\quad V_0^{\varepsilon}=x
\end{eqnarray}
while the solution $\{Y_t\}_{t\in[0,T]}$ of (\ref{ito SDE}) can be represented as
\begin{eqnarray*}
Y_t=V_t\diamond \mathcal{E}_{0}(t)^{\diamond -1}
\end{eqnarray*}
where
\begin{eqnarray}\label{equation for V 2}
\frac{dV_t}{dt}=b\left(V_t\diamond (\mathcal{E}_{0}(t))^{\diamond -1}\right)\diamond \mathcal{E}_{0}(t),\quad V_0=x.
\end{eqnarray}
Here, for $f\in L^2([0,T])$ the symbol $\mathcal{E}(f)^{\diamond -1}$ stands for the so called \emph{Wick inverse} of $\mathcal{E}(f)$ which, in this particular case, coincides with $\mathcal{E}(-f)$. The next lemma will serve to write equations (\ref{equation for V}) and (\ref{equation for V 2}) in a Wick product-free form.

\begin{lemma}\label{ito5}
If $F\in \mathcal{L}^p(W,\mu)$ for some $p>1$ and $\Psi:\mathbb{R}\rightarrow\mathbb{R}$ is measurable and with at most linear growth at infinity, then for all $h\in L^2([0,T])$ we have:
\begin{equation*}
\Psi\left(F\diamond \mathcal{E}(h)\right)\diamond\mathcal{E}(-h)=\Psi\left(F\cdot\mathcal{E}(-h)^{-1} \right)\cdot\mathcal{E}(-h).
\end{equation*}
\end{lemma}

\begin{proof}
We apply twice Gjessing's lemma (see Holden et al. \cite{HOUZ}) to get:
\begin{eqnarray*}
\Psi \left(F\diamond \mathcal{E}(h)\right)\diamond\mathcal{E}(-h)&=& \Psi\left(T_{-h}F\cdot \mathcal{E}(h)\right)\diamond\mathcal{E}(-h)\\
&=&T_{h}\left(\Psi(T_{-h}F\cdot\mathcal{E}(h))\right)\cdot\mathcal{E}(-h)\\
&=&\Psi\left(F\cdot T_{h}\mathcal{E}(h)\right)\cdot\mathcal{E}(-h) \\
&=&\Psi\left(F\cdot\mathcal{E}(-h)^{-1}\right)\cdot\mathcal{E}(-h).
\end{eqnarray*}
Here we utilized the identities
\begin{eqnarray*}
T_{h}\mathcal{E}(h)&=&\mathcal{E}(h)\exp\left\{\int_{0}^{T}h(s)^2ds\right\}\\
&=&\exp\left\{\int_0^Th(s)dB_s+\frac{1}{2}\int_{0}^{T}h(s)^2ds\right\}\\
&=&\mathcal{E}(-h)^{-1}.
\end{eqnarray*}
The proof is complete.
\end{proof}
\noindent Therefore, by Lemma \ref{ito5} we can rewrite equation (\ref{equation for V}) as
\begin{eqnarray*}
\frac{dV_t^{\varepsilon}}{dt}=b\left(V_t^{\varepsilon}\cdot(\mathcal{E}_{\varepsilon}(t))^{-1}\right)\cdot \mathcal{E}_{\varepsilon}(t)
\end{eqnarray*}
and equation (\ref{equation for V 2}) as
\begin{eqnarray*}
\frac{dV_t}{dt}=b\left(V_t\cdot(\mathcal{E}_{0}(t))^{-1}\right)\cdot \mathcal{E}_{0}(t)
\end{eqnarray*}
since, as we mentioned before,
\begin{eqnarray*}
\mathcal{E}_{\varepsilon}(t)^{\diamond -1}=\mathcal{E}(K_{\varepsilon}(t,\cdot))\quad\mbox{ and }\quad
\mathcal{E}_{0}(t)^{\diamond -1}=\mathcal{E}(1_{[0,t]}(\cdot)).
\end{eqnarray*}
The following two propositions are the stochastic exponential's counterparts of Proposition \ref{distance exponentials}.

\begin{proposition}\label{distance stochastic exponentials}
Let $f,g\in L^2([0,T])$. Then, for any $p\geq 1$ we have
\begin{eqnarray*}
\Vert\mathcal{E}(f)-\mathcal{E}(g)\Vert_p\leq C\cdot\mathcal{S}_p(|f-g|)
\end{eqnarray*}
where, as before,
\begin{eqnarray*}
\mathcal{S}_p(\lambda)=\lambda\exp\left\{p\lambda^2\right\}+\exp\{\lambda^2/2\}-1,\quad \lambda\in\mathbb{R}
\end{eqnarray*}
and $C$ is a constant depending on $p$ and $|g|$.
\end{proposition}

\begin{proof}
Let $f\in L^2([0,T])$ and $p\geq 1$; then, according to the Poincar\'e inequality (see Theorem 5.5.11 in Bogachev \cite{Bogachev}), we can write
\begin{eqnarray*}
\Vert\mathcal{E}(f)-1\Vert_p&\leq&\mathcal{C}(p)\left\| |D\mathcal{E}(f)|_{L^2([0,T])}\right\|_p\\
&=&\mathcal{C}(p)\left\| |\mathcal{E}(f)f|_{L^2([0,T])}\right\|_p\\
&=&\mathcal{C}(p)|f|\Vert\mathcal{E}(f)\Vert_p\\
&=&\mathcal{C}(p)|f|\exp\left\{\frac{p-1}{2}|f|^2\right\}
\end{eqnarray*}
where $D$ denotes the Malliavin derivative and $\mathcal{C}(p)$ is a positive constant depending only on $p$. Therefore, for any $f,g\in L^2([0,T])$ and $p\geq 1$ we have
\begin{eqnarray*}
\Vert\mathcal{E}(f)-\mathcal{E}(g)\Vert_p&=&\Vert\mathcal{E}(g)\diamond\left(\mathcal{E}(f-g)-1\right)\Vert_p\\
&\leq&\Vert\mathcal{E}(\sqrt{2}g)\Vert_{p}\Vert\mathcal{E}(\sqrt{2}(f-g))-1\Vert_{p}\\
&\leq&e^{(p-1)|g|^2}\mathcal{C}(p)|f-g|\exp\left\{(p-1)|f-g|^2\right\}\\
&\leq&C\mathcal{S}_p(|f-g|)
\end{eqnarray*}
where we utilized an inequality for the Wick product from Da Pelo et al. \cite{DLS 2011}.
\end{proof}

\begin{proposition}\label{distance inverse stochastic exponentials}
Let $f,g\in L^2([0,T])$. Then, for any $p\geq 1$ we have
\begin{eqnarray*}
\Vert\mathcal{E}(f)^{-1}-\mathcal{E}(g)^{-1}\Vert_p\leq C\cdot\mathcal{S}_p(\sqrt{2}|f-g|)
\end{eqnarray*}
where $C$ is a constant depending on $p$ and $|g|$.
\end{proposition}

\begin{proof}
Denote by $\Gamma(1/\sqrt{2})$ the bounded linear operator acting on stochastic exponentials according to the prescription
\begin{eqnarray*}
\Gamma(1/\sqrt{2})\mathcal{E}(f):=\mathcal{E}(f/\sqrt{2}).
\end{eqnarray*}
This operator coincides with the Ornstein-Uhlenbeck semigroup $\{P_t\}_{t\geq 0}$ for a proper choice of the parameter $t$ (see Janson \cite{J} for details) and therefore it is a contraction on any $\mathcal{L}^p(W,\mu)$ for $p\geq 1$. Moreover, by a direct verification one can see that
\begin{eqnarray}\label{anti wick}
\mathcal{E}(f)^{-1}=\Gamma(1/\sqrt{2})\exp\{-\delta(\sqrt{2}f)\}.
\end{eqnarray}
Hence, we can write
\begin{eqnarray*}
\Vert\mathcal{E}(f)^{-1}-\mathcal{E}(g)^{-1}\Vert_p&=&\Vert\Gamma(1/\sqrt{2})\exp\{-\delta(\sqrt{2}f)\}-
\Gamma(1/\sqrt{2})\exp\{-\delta(\sqrt{2}g)\}\Vert_p\\
&\leq&\Vert\exp\{-\delta(\sqrt{2}f)\}-\exp\{-\delta(\sqrt{2}g)\}\Vert_p.
\end{eqnarray*}
Therefore, by means of Proposition \ref{distance exponentials} we can conclude that
\begin{eqnarray*}
\Vert\mathcal{E}(f)^{-1}-\mathcal{E}(g)^{-1}\Vert_p&\leq&\Vert\exp\{-\delta(\sqrt{2}f)\}-
\exp\{-\delta(\sqrt{2}g)\}\Vert_p\\
&\leq&C\cdot\mathcal{S}_p(\sqrt{2}|f-g|).
\end{eqnarray*}
\end{proof}

\begin{remark}
The idea of the proof of the previous proposition, and in particular identity (\ref{anti wick}), is inspired by the investigation carried in Da Pelo and Lanconelli \cite{DL}, where a new probabilistic representation for the solution of the heat equation is derived in terms of the operator $\Gamma(1/\sqrt{2})$ and its inverse.
\end{remark}

\subsection{Proof of Theorem \ref{main theorem 2}}

As before, we divide the proof in two steps.\\

\noindent \textbf{Step one}: We prove that for any $p\geq 1$ there exists a positive constant $C$ (depending on $p$, $|x|$, $T$, $C_1$, $C_2$ and $M$) such that for any $q$ greater than $p$
\begin{eqnarray}\label{inequality step 1 ito}
\left\|\sup_{t\in [0,T]} |V_t^{\varepsilon}-V_t|\right\|_p\leq C\cdot\mathcal{S}_q\left(\sqrt{2}\sup_{s\in [0,T]}|K_{\varepsilon}(s,\cdot)-1_{[0,s]}(\cdot)|\right).
\end{eqnarray}
The proof can be carried following the same line of the proof of \emph{Step one} of Theorem \ref{main theorem 1}; we have simply to replace $\{Z_t\}_{t\in [0,T]}$ and $\{Z^{\varepsilon}_t\}_{t\in [0,T]}$ with $\{V_t\}_{t\in [0,T]}$ and $\{V^{\varepsilon}_t\}_{t\in [0,T]}$, respectively. Moreover, the exponentials $\{E_{\varepsilon}(t)\}_{t\in [0,T]}$ and $\{E_{0}(t)\}_{t\in [0,T]}$  have to be replaced by  $\{\mathcal{E}_{\varepsilon}(t)\}_{t\in [0,T]}$ and $\{\mathcal{E}_{0}(t)\}_{t\in [0,T]}$, respectively.
The estimate (\ref{estimate norm of Z}) changes to
\begin{eqnarray*}
\left\|\sup_{t\in [0,T]}|V_t^{\varepsilon}|\right\|_q&\leq&\left(|x|+C_2 T\exp\left\{\frac{q-1}{2}\sup_{s\in[0,T]}|K_{\varepsilon}(s,\cdot)|^2\right\} \right)e^{C_2T}.
\end{eqnarray*}
We remark that for all $r\geq 1$ we have
\begin{eqnarray*}
\Vert E_{\varepsilon}(t)\Vert_r=\exp\left\{\frac{r}{2}|K_{\varepsilon}(s,\cdot)|^2\right\}
\end{eqnarray*}
while
\begin{eqnarray*}
\Vert \mathcal{E}_{\varepsilon}(t)\Vert_r=\exp\left\{\frac{r-1}{2}|K_{\varepsilon}(s,\cdot)|^2\right\}\quad\mbox{ and }\quad\Vert\mathcal{E}_{\varepsilon}(t)^{-1}\Vert_r=\exp\left\{\frac{r+1}{2}|K_{\varepsilon}(s,\cdot)|^2\right\}.
\end{eqnarray*}
Moreover, we utilize Proposition \ref{distance stochastic exponentials} and Proposition \ref{distance inverse stochastic exponentials} with $f(\cdot)=K_{\varepsilon}(s,\cdot)$ and $g(\cdot)=1_{[0,s]}(\cdot)$ instead of Proposition \ref{distance exponentials}. \\

\noindent\textbf{Step two}: We prove that for any $p\geq 1$ there exists a positive constant $C$ (depending on $p$, $|x|$, $T$, $C_1$, $C_2$ and $M$) such that for any $q$ greater than $p$
\begin{eqnarray*}
\sup_{t\in [0,T]}\left\|Y_t^{\varepsilon}-Y_t\right\|_p\leq C\cdot\mathcal{S}_q\left(\sqrt{2}\sup_{s\in [0,T]}|K_{\varepsilon}(s,\cdot)-1_{[0,s]}(\cdot)|\right).
\end{eqnarray*}
We first note that
\begin{eqnarray*}
Y_t^{\varepsilon}-Y_t=V_t^{\varepsilon}\diamond\mathcal{E}_{\varepsilon}(t)^{\diamond -1}-V_t\diamond\mathcal{E}_{0}(t)^{\diamond -1}.
\end{eqnarray*}
To ease the readability of the formulas we will adopt, only for this part of the proof, the notation
\begin{eqnarray*}
\tilde{\mathcal{E}}_{\varepsilon}(t):=\mathcal{E}_{\varepsilon}(t)^{\diamond -1}\quad\mbox{ and }\quad\tilde{\mathcal{E}}_{0}(t):=\mathcal{E}_{0}(t)^{\diamond -1}.
\end{eqnarray*}
Then, by mean of Gjessing's Lemma we have
\begin{eqnarray*}
Y_t^{\varepsilon}-Y_t&=&V_t^{\varepsilon}\diamond\tilde{\mathcal{E}}_{\varepsilon}(t)-
V_t\diamond\tilde{\mathcal{E}}_{0}(t)\\
&=&V_t^{\varepsilon}\diamond\tilde{\mathcal{E}}_{\varepsilon}(t)-V_t^{\varepsilon}\diamond\tilde{\mathcal{E}}_{0}(t)
+V_t^{\varepsilon}\diamond\tilde{\mathcal{E}}_{0}(t)-V_t\diamond\tilde{\mathcal{E}}_{0}(t)\\
&=&T_{-K_{\varepsilon}(t,\cdot)}V_t^{\varepsilon}\cdot\tilde{\mathcal{E}}_{\varepsilon}(t)-
T_{-1_{[0,t]}(\cdot)}V_t^{\varepsilon}\cdot\tilde{\mathcal{E}}_{0}(t)+
\left(V_t^{\varepsilon}-V_t\right)\diamond\tilde{\mathcal{E}}_{0}(t)\\
&=&T_{-K_{\varepsilon} (t,\cdot)}V_t^{\varepsilon}\cdot\tilde{\mathcal{E}}_{\varepsilon}(t)-
T_{-K_{\varepsilon} (t,\cdot)}V_t^{\varepsilon}\cdot\tilde{\mathcal{E}}_{0}(t)
+T_{-K_{\varepsilon} (t,\cdot)}V_t^{\varepsilon}\cdot\tilde{\mathcal{E}}_{0}(t)\\
&&-T_{-1_{[0,t]}(\cdot)}V_t^{\varepsilon}\cdot\tilde{\mathcal{E}}_{0}(t)+
\left(V_t^{\varepsilon}-V_t\right)\diamond\tilde{\mathcal{E}}_{0}(t)\\
&=&T_{-K_{\varepsilon}(t,\cdot)}V_t^{\varepsilon}\cdot\left(\tilde{\mathcal{E}}_{\varepsilon}(t)
-\tilde{\mathcal{E}}_{0}(t)\right)
+\left(T_{-K_{\varepsilon} (t,\cdot)}V_t^{\varepsilon}-T_{-1_{[0,t]}(\cdot)}V_t^{\varepsilon}\right)\cdot\tilde{\mathcal{E}}_{0}(t)\\
&&+T_{-1_{[0,t]}(\cdot)}\left(V_t^{\varepsilon}-V_t\right)\cdot\tilde{\mathcal{E}}_{0}(t)\\
&=&\mathcal{F}_1+\mathcal{F}_2+\mathcal{F}_3
\end{eqnarray*}
where we set
\begin{eqnarray*}
\mathcal{F}_1:=T_{-K_{\varepsilon}(t,\cdot)}V_t^{\varepsilon}\cdot\left(\tilde{\mathcal{E}}_{\varepsilon}(t)
-\tilde{\mathcal{E}}_{0}(t)\right)\quad\quad\mathcal{F}_2:=\left(T_{-K_{\varepsilon} (t,\cdot)}V_t^{\varepsilon}-T_{-1_{[0,t]}(\cdot)}V_t^{\varepsilon}\right)\cdot\tilde{\mathcal{E}}_{0}(t)
\end{eqnarray*}
and
\begin{eqnarray*}
\mathcal{F}_3:=T_{-1_{[0,t]}(\cdot)}\left(V_t^{\varepsilon}-V_t\right)\cdot\tilde{\mathcal{E}}_{0}(t).
\end{eqnarray*}
Hence, for any $p\geq 1$ we can write
\begin{eqnarray*}
\Vert Y_t^{\varepsilon}-Y_t\Vert_p&\leq&\Vert \mathcal{F}_1\Vert_p+\Vert \mathcal{F}_2\Vert_p+\Vert \mathcal{F}_3\Vert_p.
\end{eqnarray*}
We recall (see Theorem 14.1 in Janson \cite{J}) that for any $g\in L^2([0,T])$ the linear operator $T_g$ is bounded from $\mathcal{L}^q(W,\mu)$ to $\mathcal{L}^p(W,\mu)$ for any $p<q$. Therefore, by the H\"older inequality and Proposition \ref{distance stochastic exponentials} we deduce
\begin{eqnarray*}
\Vert\mathcal{F}_1\Vert_p&=&\left\| T_{-K_{\varepsilon}(t,\cdot)}V_t^{\varepsilon}\cdot\left(\tilde{\mathcal{E}}_{\varepsilon}(t)
-\tilde{\mathcal{E}}_{0}(t)\right)\right\|_p\\
&\leq&\left\| T_{-K_{\varepsilon}(t,\cdot)}V_t^{\varepsilon}\right\|_{q_1}\cdot\left\|\tilde{\mathcal{E}}_{\varepsilon}(t)
-\tilde{\mathcal{E}}_{0}(t)\right\|_{q_2}\\
&\leq& C \left\|V_t^{\varepsilon}\right\|_r\cdot\left\|\tilde{\mathcal{E}}_{\varepsilon}(t)
-\tilde{\mathcal{E}}_{0}(t)\right\|_{q_2}\\
&\leq& C\cdot\mathcal{S}_{q_2}\left(\sqrt{2}\sup_{s\in [0,T]}|K_{\varepsilon}(s,\cdot)-1_{[0,s]}(\cdot)|\right).
\end{eqnarray*}
where $p<q_1<r$, $C$ is a constant depending on the parameters appearing in the statement of the theorem and $\frac{1}{p}=\frac{1}{q_1}+\frac{1}{q_2}$. The term $\Vert\mathcal{F}_3\Vert_p$ is treated similarly with the help of inequality (\ref{inequality step 1 ito}). Let us now focus on $\Vert\mathcal{F}_2\Vert_p$. We first observe that
\begin{eqnarray*}
\Vert\mathcal{F}_2\Vert_p&=&\left\Vert \left(T_{-K_{\varepsilon} (t,\cdot)}V_t^{\varepsilon}-T_{-1_{[0,t]}(\cdot)}V_t^{\varepsilon}\right)\cdot\tilde{\mathcal{E}}_{0}(t) \right\Vert_p\\
&\leq&\left\Vert T_{-K_{\varepsilon} (t,\cdot)}V_t^{\varepsilon}-T_{-1_{[0,t]}(\cdot)}V_t^{\varepsilon}\right\Vert_q
\cdot\left\Vert\tilde{\mathcal{E}}_{0}(t)\right\Vert_r.
\end{eqnarray*}
According to Theorem 14.1 in Janson \cite{J} the map $T_gX$ is jointly continuous in the variables $(g,X)$ from $L^2([0,T])\times\mathcal{L}^q(W,\mu)$ to $\mathcal{L}^p(W,\mu)$ for $p<q$. Therefore, the first term in the last member of the previous inequality tends to zero as $\varepsilon\to 0^+$. However, we need to know the speed of such convergence. The following lemma will help us in this direction.
\begin{lemma}\label{Malliavin}
For any $X\in\mathbb{D}^{1,q}$ and $h\in L^2([0,T])$ with $|h|<\delta$ one has
\begin{eqnarray*}
\Vert T_hX-X\Vert_p\leq C|h|\Vert X\Vert_{\mathbb{D}^{1,q}}
\end{eqnarray*}
where $p<q$ and $C$ depends on $\delta$, $p$ and $q$.
\end{lemma}

\begin{proof}
Since the linear span of the stochastic exponentials is dense in $\mathcal{L}^p(W,\mu)$ and in $\mathbb{D}^{1,q}$, we will prove the lemma for $X=\sum_{j=1}^n\alpha_j\mathcal{E}(f_j)$ where $\alpha_1,...,\alpha_n\in\mathbb{R}$ and $f_1,...,f_n\in L^2([0,T])$. By the mean value theorem we can write for $\theta\in [0,1]$ that
\begin{eqnarray*}
T_h\sum_{j=1}^n\alpha_j\mathcal{E}(f_j)-\sum_{j=1}^n\alpha_j\mathcal{E}(f_j)&=&
\sum_{j=1}^n\alpha_j\mathcal{E}(f_j)\left( e^{\langle h,f_j\rangle}-1\right)\\
&=&\sum_{j=1}^n\alpha_j\mathcal{E}(f_j)e^{\theta\langle h,f_j\rangle}\langle h,f_j\rangle\\
&=&T_{\theta h}D_h\sum_{j=1}^n\alpha_j\mathcal{E}(f_j)
\end{eqnarray*}
where $D_{\theta h}$ stands for the Malliavin derivative in the direction $\theta h$. We now take the $\mathcal{L}^p(W,\mu)$ norm to get
\begin{eqnarray*}
\left\|T_h\sum_{j=1}^n\alpha_j\mathcal{E}(f_j)-\sum_{j=1}^n\alpha_j\mathcal{E}(f_j)\right\|_p&=&
\left\|T_{\theta h}D_h\sum_{j=1}^n\alpha_j\mathcal{E}(f_j)\right\|_p\\
&\leq&C(h)\left\|D_h\sum_{j=1}^n\alpha_j\mathcal{E}(f_j)\right\|_q\\
&\leq&C(h)|h|\left\|\left|D\sum_{j=1}^n\alpha_j\mathcal{E}(f_j)\right|_{L^2([0,T])}\right\|_q\\
&\leq&C(h)|h|\left\|\sum_{j=1}^n\alpha_j\mathcal{E}(f_j)\right\|_{\mathbb{D}^{1,q}}.
\end{eqnarray*}
\end{proof}
\noindent We now continue the analysis of the term
\begin{eqnarray*}
\left\Vert T_{-K_{\varepsilon (t,\cdot)}}V_t^{\varepsilon}-T_{-1_{[0,t]}(\cdot)}V_t^{\varepsilon}\right\Vert_q.
\end{eqnarray*}
It is not difficult to see that Assumption \ref{assumption on b and sigma} implies that for any $\varepsilon>0$ and $t\in[0,T]$ the random variable $V_t^{\varepsilon}$ belongs to $\mathbb{D}^{1,q}$ for all $q\geq 1$. Moreover, the $\mathbb{D}^{1,q}$-norm of $V_t^{\varepsilon}$ is bounded uniformly with respect to $\varepsilon$ (observe that $V_t$, which corresponds to the case $\varepsilon=0$, is related to an It\^o type SDE which possesses the required smoothness). Therefore,
\begin{eqnarray*}
\left\Vert T_{-K_{\varepsilon} (t,\cdot)}V_t^{\varepsilon}-T_{-1_{[0,t]}(\cdot)}V_t^{\varepsilon}\right\Vert_q&=&
\left\Vert T_{-1_{[0,t]}(\cdot)}\left(T_{1_{[0,t]}(\cdot)-K_{\varepsilon} (t,\cdot)}V_t^{\varepsilon}-V_t^{\varepsilon}\right)\right\Vert_q\\
&\leq&C\left\Vert T_{1_{[0,t]}(\cdot)-K_{\varepsilon} (t,\cdot)}V_t^{\varepsilon}-V_t^{\varepsilon}\right\Vert_r\\
&\leq&C|K_{\varepsilon} (t,\cdot)-1_{[0,t]}(\cdot)|\Vert V_t^{\varepsilon}\Vert_{\mathbb{D}^{1,r}}.
\end{eqnarray*}
for $r>q$. Combining all the estimates above we conclude
\begin{eqnarray*}
\Vert Y_t^{\varepsilon}-Y_t\Vert_p&\leq&\Vert \mathcal{F}_1\Vert_p+\Vert \mathcal{F}_2\Vert_p+\Vert \mathcal{F}_3\Vert_p\\
&\leq& C\left(\mathcal{S}_q\left(\sqrt{2}\sup_{s\in [0,T]}|K_{\varepsilon}(s,\cdot)-1_{[0,s]}(\cdot)|\right)+|K_{\varepsilon} (t,\cdot)-1_{[0,t]}(\cdot)|\right)\\
&\leq&C\left(\mathcal{S}_q\left(\sqrt{2}\sup_{s\in [0,T]}|K_{\varepsilon}(s,\cdot)-1_{[0,s]}(\cdot)|\right)+\sup_{t\in[0,T]}|K_{\varepsilon} (t,\cdot)-1_{[0,t]}(\cdot)|\right)\\
&\leq&C\mathcal{S}_q\left(\sqrt{2}\sup_{s\in [0,T]}|K_{\varepsilon}(s,\cdot)-1_{[0,s]}(\cdot)|\right).
\end{eqnarray*}
The proof of Theorem \ref{main theorem 2} is now complete.

\section{Proof of Theorem \ref{main theorem 3}}

We first note that the solution $\{A_t\}_{t\in[0,T]}$ of
\begin{eqnarray*}
\frac{dA_t}{dt}=b(A_t)+ A_t\cdot g(t)\quad A_0=x,
\end{eqnarray*}
where $g:[0,T]\rightarrow\mathbb{R}$ is a continuous function,
can be represented as
\begin{eqnarray}\label{ito5,5}
A_t=G_t\cdot\exp\left\{\int_0^tg(s)ds\right\}
\end{eqnarray}
where $\{G_t\}_{t\in[0,T]}$ solves
\begin{eqnarray}\label{ito6}
\frac{dG_t}{dt}=b\left(G_t\cdot\exp\left\{\int_0^t g(s)ds\right\}\right)\cdot\exp\left\{-\int_0^t g(s) ds\right\}.
\end{eqnarray}
Moreover, recalling the argument from the previous section, we know that the solution $\{Y_t^{\varepsilon}\}_{t\in[0,T]}$ of (\ref{approx ito}) can be represented as
\begin{eqnarray*}
Y_t^{\varepsilon}=V_t^{\varepsilon}\diamond \mathcal{E}_{\varepsilon}(t)^{\diamond -1}
\end{eqnarray*}
where
\begin{eqnarray*}
\frac{dV_t^{\varepsilon}}{dt}=b\left(V_t^{\varepsilon}\cdot(\mathcal{E}_{\varepsilon}(t))^{-1}\right)\cdot \mathcal{E}_{\varepsilon}(t).
\end{eqnarray*}
Since by definition
\begin{eqnarray*}
(\mathcal{E}_{\varepsilon}(t))^{-1}&=&\exp\left\{\int_0^TK_{\varepsilon}(t,s)dB_s+\frac{1}{2}|K_{\varepsilon}(t,\cdot)|^2\right\}\\
&=&\exp\left\{B_t^{\varepsilon}+\frac{1}{2}|K_{\varepsilon}(t,\cdot)|^2\right\}\\
&=&\exp\left\{\int_0^t\left(\frac{d B_s^{\varepsilon}}{ds}+\frac{1}{2}\frac{d|K_{\varepsilon}(s,\cdot)|^2}{ds}\right)ds\right\}
\end{eqnarray*}
a comparison with (\ref{ito5,5}) and (\ref{ito6}) shows that, by choosing
\begin{eqnarray*}
g(t)=\frac{1}{2}\frac{d}{dt}|K_{\varepsilon}(t,\cdot)|^{2}+\frac{dB_{t}^{\varepsilon}}{dt},
\end{eqnarray*}
we can write
\begin{equation*}
V_{t}^{\varepsilon}=S_{t}^{\varepsilon}\cdot\mathcal{E}_{\varepsilon}(t)
\end{equation*}
where $\{S^{\varepsilon}_t\}_{t\in[0,T]}$ is the process defined in the statement of Theorem \ref{main theorem 3}. Therefore,
\begin{eqnarray*}
Y^{\varepsilon}_t&=&V_{t}^{\varepsilon}\diamond\mathcal{E}_{\varepsilon}(t)^{\diamond -1}\\
&=&\left(S_{t}^{\varepsilon}\cdot\mathcal{E}_{\varepsilon}(t)\right)\diamond \mathcal{E}_{\varepsilon}(t)^{\diamond -1}\\
&=& T_{-K_{\varepsilon}(t,\cdot)}\left(S_{t}^{\varepsilon}\cdot\mathcal{E}_{\varepsilon}(t)\right)\cdot \mathcal{E}_{\varepsilon}(t)^{\diamond -1}\\
&=&T_{-K_{\varepsilon}(t,\cdot)}S_{t}^{\varepsilon}\cdot T_{-K_{\varepsilon}(t,\cdot)}\mathcal{E}_{\varepsilon}(t)\cdot \mathcal{E}_{\varepsilon}(t)^{\diamond -1}\\
&=& T_{-K_{\varepsilon}(t,\cdot)}S_{t}^{\varepsilon}\cdot\exp\left\{-\int_{0}^{T}K_{\varepsilon}(t,s)dB_{s}-
\frac{1}{2}|K_{\varepsilon}(t,\cdot)|^{2}+|K_{\varepsilon}(t,\cdot)|^{2}\right\}\cdot \mathcal{E}_{\varepsilon}(t)^{\diamond -1}\\
&=&T_{-K_{\varepsilon}(t,\cdot)}S_{t}^{\varepsilon}.
\end{eqnarray*}
Here, in the third equality, we utilized Gjessing Lemma. The proof of Theorem \ref{main theorem 3} is complete.

\subsection{Alternative proof}

We are now going to prove a technical result of independent interest that will be used to obtain a different and more direct proof of Theorem \ref{main theorem 3}.

\begin{proposition}\label{îto9}
Let $\{X_t\}_{t\in[0,T]}$ be a stochastic process such that:
\begin{itemize}
\item the function $t\mapsto X_t$ is differentiable
\item the random variable $X_t$ belongs to $\mathcal{L}^p(W,\mu)$ for some $p>1$ and all $t\in [0,T]$.
\end{itemize}
If the function $h:[0,T]^2\rightarrow\mathbb{R}$ is such that
\begin{itemize}
\item for almost all $s\in[0,T]$ the function $t\mapsto h(t,s)$ is continuously differentiable
\item for all $t\in [0,T]$ the functions $h(t,\cdot)$ and $\partial_t h(t,\cdot)$ belong to $L^2([0,T])$
\end{itemize}
then
\begin{eqnarray*}
\frac{d}{dt}(T_{h(t,\cdot)}X_t)&=&T_{h(t,\cdot)}\frac{dX_t}{dt}+T_{h(t,\cdot)}X_t \cdot\int_0^T \partial_th(t,s)dB_s\\
&&-T_{h(t,\cdot)}X_t\diamond\int_0^T\partial_th(t,s)dB_s.
\end{eqnarray*}
\end{proposition}

\begin{proof}
To simplify the notation we set
\begin{eqnarray*}
\delta(h(t,\cdot)):=\int_0^Th(t,s)dB_s\quad\mbox{ and }\quad\delta(\partial_th(t,\cdot)):=\int_0^T\partial_th(t,s)dB_s.
\end{eqnarray*}
According to Gjessing Lemma we know that
\begin{eqnarray*}
T_{h(t,\cdot)}X_t\diamond \mathcal{E}(h(t,\cdot))=X_t\cdot\mathcal{E}(h(t,\cdot))
\end{eqnarray*}
or equivalently,
\begin{eqnarray}\label{Gjessing}
T_{h(t,\cdot)}X_t=\left(X_t\cdot\mathcal{E}(h(t,\cdot))\right)\diamond\mathcal{E}(-h(t,\cdot)).
\end{eqnarray}
We now use the chain rule for the Wick product to get
\begin{eqnarray*}
\frac{d}{dt}(T_{h(t,\cdot)}X_t)&=&\frac{d}{dt}\left(X_t\cdot\mathcal{E}(h(t,\cdot))\right)
\diamond\mathcal{E}(-h(t,\cdot))+(X_t\cdot\mathcal{E}(h(t,\cdot)))\diamond\frac{d}{dt}\mathcal{E}(-h(t,\cdot))\\
&=&\left(\frac{dX_t}{dt}\cdot\mathcal{E}(h(t,\cdot))\right)\diamond \mathcal{E}(-h(t,\cdot))+\left(X_t\cdot\frac{d}{dt}\mathcal{E}(h(t,\cdot))\right)\diamond\mathcal{E}(-h(t,\cdot))\\
&&+\left(X_t\cdot\mathcal{E}(h(t,\cdot))\right)\diamond\frac{d}{dt}\mathcal{E}(-h(t,\cdot))\\
&=&\left(\frac{dX_t}{dt}\cdot\mathcal{E}(h(t,\cdot))\right)\diamond\mathcal{E}(-h(t,\cdot))\\
&&+\left(X_t\cdot\mathcal{E}(h(t,\cdot))\cdot\frac{d}{dt}\left(\delta(h(t,\cdot))-\frac{1}{2}|h(t,\cdot)|^2\right)\right)
\diamond\mathcal{E}(-h(t,\cdot))\\
&&+\left(X_t\cdot\mathcal{E}(h(t,\cdot))\right)\diamond\mathcal{E}(-h(t,\cdot))
\diamond\frac{d}{dt}\delta(-h(t,\cdot))\\
\end{eqnarray*}
Observe that according to identity (\ref{Gjessing}) we can write
\begin{eqnarray*}
\left(\frac{dX_t}{dt}\cdot\mathcal{E}(h(t,\cdot))\right)\diamond\mathcal{E}(-h(t,\cdot))=T_{h(t,\cdot)}\frac{dX_t}{dt}.
\end{eqnarray*}
Therefore, the last chain of equalities becomes
\begin{eqnarray*}
\frac{d}{dt}(T_{h(t,\cdot)}X_t)&=&T_{h(t,\cdot)}\frac{dX_t}{dt}\\
&&+\left(X_t\cdot\mathcal{E}(h(t,\cdot))
\cdot(\delta(\partial_th(t,\cdot))-\langle h(t,\cdot),\partial_t h(t,\cdot)\rangle)\right)
\diamond\mathcal{E}(-h(t,\cdot))\\
&&-T_{h(t,\cdot)}X_t\diamond\delta(\partial_th(t,\cdot))\\
&=&T_{h(t,\cdot)}\frac{dX_t}{dt}+T_{h(t,\cdot)}\left(X_t\cdot(\delta(\partial_th(t,\cdot))-\langle h(t,\cdot),\partial_t h(t,\cdot)\rangle)\right)\\
&&-T_{h(t,\cdot)}X_t\diamond\delta(\partial_th(t,\cdot))\\
&=&T_{h(t,\cdot)}\frac{dX_t}{dt}+T_{h(t,\cdot)}X_t\cdot T_{h(t,\cdot)}(\delta(\partial_th(t,\cdot))-\langle h(t,\cdot),\partial_t h(t,\cdot)\rangle)\\
&&-T_{h(t,\cdot)}X_t\diamond\delta(\partial_th(t,\cdot))\\
&=&T_{h(t,\cdot)}\frac{dX_t}{dt}+T_{h(t,\cdot)}X_t\cdot\delta(\partial_th(t,\cdot))
-T_{h(t,\cdot)}X_t\diamond\delta(\partial_th(t,\cdot)).
\end{eqnarray*}
The proof is complete.
\end{proof}

By means of Proposition \ref{îto9}, we are now able to prove identity (\ref{ito}) from Theorem \ref{main theorem 3} via a direct verification. More precisely, let $\{S_{t}^{\varepsilon}\}_{t\in[0,T]}$ be the process in the statement of Theorem \ref{main theorem 3}. Then, using equation (\ref{ito-stra SDE}) we get
\begin{eqnarray*}
\frac{d}{dt}T_{-K_\varepsilon(t,\cdot)}S_{t}^{\varepsilon}&=& T_{-K_\varepsilon(t,\cdot)}\frac{dS_{t}^{\varepsilon}}{dt}-T_{-K_\varepsilon(t,\cdot)}S_{t}^{\varepsilon}
\cdot\int_{0}^{T}\partial_tK_\varepsilon(t,s)dB_s\\
&&+T_{-K_\varepsilon(t,\cdot)}S_{t}^{\varepsilon}\diamond\int_{0}^{T}\partial_tK_\varepsilon(t,s)dB_s\\
&=&T_{-K_\varepsilon(t,\cdot)}\left(b(S_{t}^{\varepsilon})+\frac{1}{2}\frac{d}{dt}|K_{\varepsilon}(t,\cdot)|^{2} \cdot S_{t}^{\varepsilon}+S_{t}^{\varepsilon}\cdot\frac{dB_{t}^{\varepsilon}}{dt}\right)\\
&&-T_{-K_\varepsilon(t,\cdot)}S_{t}^{\varepsilon}\cdot\frac{dB_{t}^{\varepsilon}}{dt}
+T_{-K_\varepsilon(t,\cdot)}S_{t}^{\varepsilon}\diamond\frac{dB_{t}^{\varepsilon}}{dt}\\
&=&b\left(T_{-K_\varepsilon(t,\cdot)}S_{t}^{\varepsilon}\right)+
\frac{1}{2}\frac{d}{dt}|K_{\varepsilon}(t,\cdot)|^{2}\cdot T_{-K_\varepsilon(t,\cdot)}S_{t}^{\varepsilon}\\
&&+T_{-K_\varepsilon(t,\cdot)}S_{t}^{\varepsilon}\cdot\left(\frac{dB_{t}^{\varepsilon}}{dt}
-\int_0^T\partial_tK_{\varepsilon}(t,s) K_{\varepsilon}(t,s)ds\right)\\ &&-T_{-K_\varepsilon(t,\cdot)}S_{t}^{\varepsilon}\cdot\frac{dB_{t}^{\varepsilon}}{dt}+T_{-K_\varepsilon(t,\cdot)}S_{t}^{\varepsilon}\diamond\frac{dB_{t}^{\varepsilon}}{dt}\\
&=&b\left(T_{-K_\varepsilon(t,\cdot)}S_{t}^{\varepsilon}\right)+T_{-K_\varepsilon(t,\cdot)}S_{t}^{\varepsilon}\diamond\frac{dB_{t}^{\varepsilon}}{dt}.
\end{eqnarray*}
This is imply that $\{T_{-K_\varepsilon(t,\cdot)}S_{t}^{\varepsilon}\}_{t\in [0,T]}$ solves equation (\ref{approx ito}).

\end{document}